\documentclass[12pt,reqno]{amsart}

\usepackage{amssymb}

\usepackage[all]{xy}

\numberwithin{equation}{section}

\newtheorem{theorem}{Theorem}[section]
\newtheorem{proposition}[theorem]{Proposition}
\newtheorem{lemma}[theorem]{Lemma}
\newtheorem{corollary}[theorem]{Corollary}

\theoremstyle{definition}

\begin{document}

\baselineskip=15pt

\title[On Higgs bundles twisted by a vector bundle]{Remarks on Higgs bundles twisted by a vector bundle}

\author[D. Alfaya]{David Alfaya}

\address{Department of Applied Mathematics and Institute for Research in Technology, ICAI School of Engineering,
Comillas Pontifical University, C/Alberto Aguilera 25, 28015 Madrid, Spain}

\email{dalfaya@comillas.edu}

\author[I. Biswas]{Indranil Biswas}

\address{Department of Mathematics, Shiv Nadar University, NH91, Tehsil Dadri,
Greater Noida, Uttar Pradesh 201314, India}

\email{indranil.biswas@snu.edu.in, indranil29@gmail.com}

\author[P. Kumar]{Pradip Kumar}

\address{Department of Mathematics, Shiv Nadar University, NH91, Tehsil Dadri,
Greater Noida, Uttar Pradesh 201314, India}

\email{Pradip.Kumar@snu.edu.in}



\date{}

\begin{abstract}
For any $V$--twisted Higgs bundle on a compact Riemann surface $X$, where $V$ is a holomorphic vector bundle
of rank two on $X$, there are two associated Higgs bundles on $X$, twisted by line bundles, which are
constructed using a Hecke transformation
on $V$. We characterize all such pairs of Higgs bundles (twisted by line bundles) given by $V$--twisted Higgs
bundles. Using this characterization, we provide a spectral correspondence for the moduli space, identifying
$V$--twisted Higgs bundles with the direct images of certain rank one torsionfree Higgs sheaves twisted
by a line bundle on a spectral covering of the curve $X$.
\end{abstract}

\maketitle

\section{Introduction}

This work was inspired by \cite{GGN}. In \cite{GGN} the notion of $V$--twisted Higgs bundles on a compact
Riemann surface $X$, where $V$ is a holomorphic vector bundle on $X$ of rank two, was introduced. Their
moduli was constructed in \cite{Si} and \cite{GGN} and some properties of the moduli space were investigated there.

Our aim here is to point out that using a Hecke transformation on $V$, there is a natural bijective
correspondence between the $V$--twisted Higgs bundles on $X$ and a certain class of pairs of
$\mathcal S$--twisted Higgs bundles and $\mathcal L$--twisted Higgs bundles on $X$. Here $\mathcal S$
and $\mathcal L$ are fixed holomorphic line bundles on $X$ such that ${\mathcal S}\oplus {\mathcal L}$ is
obtained by a forward Hecke transformation of $V$. (See Theorem \ref{thm1} and Corollary
\ref{cor2}.)

{}From this correspondence an analogue of the Beauville-Narasimhan-Ramanan spectral correspondence for 
$V$--twisted Higgs bundles is obtained (see Theorem \ref{thm2}). Given a $V$--twisted Higgs bundle
consider the corresponding $\mathcal{S}$--twisted Higgs bundle. This $\mathcal{S}$--twisted Higgs
field gives a spectral curve $\mathcal{S} \,\supset\, Y\, \stackrel{\phi}{\longrightarrow}\, X$.
The class of $V$--twisted Higgs bundles such that the spectral curve (for the $\mathcal{S}$--twisted Higgs
field) is integral, is identified 
with the direct image of a $\varphi^*\mathcal L$--twisted torsionfree Higgs-sheaf of rank 1 on the
corresponding spectral curve.

\section*{Acknowledgements}

D.A. was supported by grants PID2022-142024NB-I00 and RED2022-134463-T funded by 
MCIN/AEI/10.13039/501100011033. I.B. is partially supported by a J. C. Bose Fellowship (JBR/2023/000003).

\section{Decomposition of Higgs fields}
\label{section:decomposition}

Let $X$ be a compact connected Riemann surface. Fix a holomorphic vector bundle $V$ on $X$ of rank two. We 
will start by showing that we can always construct two holomorphic line bundles $\mathcal{S}$ and $\mathcal{L}$ on $X$ 
and a torsion sheaf $\mathbb{T}$ supported on a finite set of distinct points of $X$ such that $V$ fits in 
a short exact sequence, of ${\mathcal O}_X$--coherent sheaves, of the form
\begin{equation}\label{e5intro}
0\, \longrightarrow\, V \,{\longrightarrow}\, \mathcal{S}\oplus \mathcal{L}\,
{\longrightarrow}\,{\mathbb T} \, \longrightarrow\, 0;
\end{equation}
in other words, $V$ is obtained from $\mathcal{S}\oplus \mathcal{L}$ by performing a Hecke transformation.
To do so, first choose a holomorphic line subbundle
\begin{equation}\label{e0}
S\ \subset\ V^*.
\end{equation}
To construct such a line subbundle explicitly, note that the Riemann--Roch theorem says that
the vector bundle $V^*\otimes L'$ has a nonzero holomorphic section for any holomorphic line
bundle $L'$ with $2\cdot\text{degree}(L')\, \geq \, 2\cdot \text{genus}(X)+ \text{degree}(V)-1$. Any
nonzero section $s$ of $V^*\otimes L'$ generates a line subbundle $L^s$ of $V^*\otimes L'$; in
other words, $L^s$ is the inverse image, in $V^*\otimes L'$, of the torsion part of $(V^*\otimes L')/
{\rm image}(s)$ under the quotient map $V^*\otimes L'\, \longrightarrow\,(V^*\otimes L')/
{\rm image}(s)$. So the line bundle $S\,=\,L^s\otimes (L')^*$ is a line subbundle of $V^*\otimes L'
\otimes (L')^*\,=\,V^*$. Consider the short exact sequence
\begin{equation}\label{e1}
0\, \longrightarrow\, S \,\longrightarrow\, V^*\, \stackrel{q}{\longrightarrow}\, V^*/S \, \longrightarrow\, 0,
\end{equation}
where $S$ is the line subbundle in \eqref{e0}.
Take a reduced effective divisor ${\mathbb D}\,=\, \sum_{i=1}^d x_i$, where $x_1,\,\cdots ,\, x_d$ are distinct
points of $X$. Thus, $(V^*/S)\otimes {\mathcal O}_X(-\mathbb{D})$ is a subsheaf of $V^*/S$. From \eqref{e1}
we have the short exact sequence
\begin{equation}\label{e2}
0\, \longrightarrow\, S \,\longrightarrow\, W\, :=\, q^{-1}((V^*/S)\otimes {\mathcal O}_X(-\mathbb{D}))\,
\stackrel{q'}{\longrightarrow}\, (V^*/S)\otimes {\mathcal O}_X(-\mathbb{D}) \, \longrightarrow\, 0,
\end{equation}
where $q$ is the quotient map in \eqref{e1} and $q'$ in \eqref{e2} is the restriction of $q$. If we choose $\mathbb{D}$ such that
\begin{equation}\label{e2b}
d \ > \ \text{degree}(V^*/S)-\text{degree}(S) +2(\text{genus}(X)-1),
\end{equation}
then $\text{degree}(S\otimes ((V^*/S)\otimes {\mathcal O}_X(-\mathbb{D}))^*) \, >\, 2(\text{genus}(X)-1)$, so
$H^1(X,\, S\otimes ((V^*/S)\otimes {\mathcal O}_X(-\mathbb{D}))^*)\,=\, 0$; indeed,
for any holomorphic line bundle $L'$ on $X$ with $\text{degree}(L')\, >\,2(\text{genus}(X)-1)$,
by Serre duality $H^1(X,\, L')\,=\, H^0(X,\, (L')^*\otimes K_X)^*\,=\, 0$. Hence \eqref{e2} splits holomorphically if \eqref{e2b} holds.
Assume that $\mathbb{D}$ is chosen so that \eqref{e2b} holds. Fix a holomorphic splitting of \eqref{e2}. This implies
that we have a decomposition
\begin{equation}\label{e3}
W\, :=\, q^{-1}((V^*/S)\otimes {\mathcal O}_X(-\mathbb{D}))\, =\,
S\oplus ((V^*/S)\otimes {\mathcal O}_X(-\mathbb{D})).
\end{equation}
Let $L\, \subset\, V^*$ be the line subbundle generated by the subsheaf $(V^*/S)\otimes
{\mathcal O}_X(-\mathbb{D})\, \subset\, W \, \subset\, V^*$ (see \eqref{e3}). In other
words, $L$ is the inverse image of the torsion part of $V^*/((V^*/S)\otimes {\mathcal O}_X(-\mathbb{D}))$
under the quotient map $V^*\, \longrightarrow\, V^*/((V^*/S)\otimes {\mathcal O}_X(-\mathbb{D}))$.
Consequently, we have a short exact sequence
\begin{equation}\label{e4}
0\, \longrightarrow\, S\oplus L \,\stackrel{\Phi^*}{\longrightarrow}\, V^*\,
\longrightarrow\,{\mathbb T}' \, \longrightarrow\, 0,
\end{equation}
where ${\mathbb T}'$ is a torsion sheaf whose support is contained in the reduced divisor $\mathbb{D}$, and
$\Phi^*$ is the natural inclusion map (the notation $\Phi^*$ is being used because its adjoint will
be used more often). Let $D$ denote the support of $\mathbb{T}'$. For each point $y\in D$,
the dimension of ${\mathbb T}'_y\,=\,
{\mathbb T}'/(\mathbb{T}'\otimes {\mathcal O}_X(-y))$ is $1$. Taking the dual, from \eqref{e4} we
have a short exact sequence
\begin{equation}\label{e5}
0\, \longrightarrow\, V \,\stackrel{\Phi}{\longrightarrow}\, S^*\oplus L^*\,
\stackrel{\xi}{\longrightarrow}\,{\mathbb T} \, \longrightarrow\, 0,
\end{equation}
where $\Phi\,:=\, (\Phi^*)^*$ is the dual of $\Phi^*$; so ${\mathbb T}$ is a torsion sheaf whose support coincides
with the support of ${\mathbb T}'$. In
fact, ${\mathbb T}$ and ${\mathbb T}'$ are isomorphic, but there is no natural isomorphism between them.

Taking $\mathcal{S}$ and $\mathcal{L}$ as $S^*$ and $L^*$ respectively, \eqref{e5} becomes the sought exact sequence
\begin{equation}\label{e6}
0\, \longrightarrow\, V \,\stackrel{\Phi}{\longrightarrow}\, {\mathcal S}\oplus {\mathcal L}\,
\stackrel{\xi}{\longrightarrow}\,{\mathbb T} \, \longrightarrow\, 0.
\end{equation}

{}From the construction of \eqref{e5} it follows immediately that for any point $x$ in the
support of the torsion sheaf $\mathbb T$, the image of the homomorphism
$$
\Phi_x\, :\, V_x \, \longrightarrow\, ({\mathcal S}\oplus {\mathcal L})_x\,=\,
{\mathcal S}_x\oplus {\mathcal L}_x
$$
has dimension $1$ and the restrictions of the map $\xi$ to the fibers $\mathcal{L}_x$ and ${\mathcal S}_x$
are both injective. Thus, the following composition of homomorphisms is an isomorphism:
\begin{equation}
\label{e6b1}
\Phi_x(V_x) \, \hookrightarrow\, {\mathcal S}_x\oplus {\mathcal L}_x\,\longrightarrow\, 
{\mathcal S}_x,
\end{equation}
where ${\mathcal S}_x\oplus {\mathcal L}_x\,\longrightarrow\, {\mathcal S}_x$ is the natural projection.
Denote the isomorphism in \eqref{e6b1} by $\rho^x_1\, :\, \Phi_x(V_x) \, \longrightarrow\, {\mathcal
S}_x$. Analogously, let $\rho^x_2\, :\, \Phi_x(V_x) \, \longrightarrow\, {\mathcal L}_x$ denote the
composition of homomorphisms
\begin{equation}
\label{e6b2}
\Phi_x(V_x) \, \hookrightarrow\, {\mathcal S}_x\oplus
{\mathcal L}_x\,\longrightarrow\, {\mathcal L}_x,
\end{equation}
where ${\mathcal S}_x\oplus {\mathcal L}_x
\,\longrightarrow\, {\mathcal L}_x$ is the other natural projection. Then, we have a natural morphism
\begin{equation}\label{e6c}
\rho^x\ :=\, \rho^x_2\circ (\rho^x_1)^{-1}\ :\ {\mathcal S}_x \ \longrightarrow\ {\mathcal L}_x.
\end{equation}

Let us also consider the homomorphisms
\begin{equation}\label{e6d}
\widehat{\rho}_1\, :=\, p_1\circ \Phi\, :\, V\, \longrightarrow\, {\mathcal S}\ \ \text{ and }
\ \ \widehat{\rho}_2\,:=\, p_2\circ\Phi :\, V\, \longrightarrow\, {\mathcal L},
\end{equation}
where $p_1$ and $p_2$ are the projections of ${\mathcal S}\oplus {\mathcal L}$ to ${\mathcal S}$ and
${\mathcal L}$ respectively, and $\Phi$ is the homomorphism in \eqref{e6}. Note that
$(\widehat{\rho}_1)_x\,=\, \rho^x_1 \circ\Phi_x$ and $(\widehat{\rho}_2)_x\,=\,\rho^x_2\circ\Phi_x$, where
$\rho^x_1$ and $\rho^x_2$ are the compositions in \eqref{e6b1} and \eqref{e6b2} respectively.

Take a holomorphic vector bundle $E$ on $X$. A $V$--\textit{twisted Higgs field}
on $E$ is a holomorphic section
$$
\theta \ \in\ H^0(X,\, \text{End}(E)\otimes V)
$$
such that $\theta\bigwedge\theta\, =\, 0$ (see \cite[Definition 2.1]{GGN}); note that
$\theta\bigwedge\theta \, \in\, H^0(X,\, \text{End}(E)\otimes\bigwedge^2 V)$.
Using the homomorphisms in \eqref{e6d}, a $V$--twisted Higgs field $\theta$ on $E$ gives sections
\begin{equation}\label{e7}
\theta_1\,:=\, ({\rm Id}_E\otimes\widehat{\rho}_1)(\theta) \, \in\, H^0(X,\, \text{End}(E)\otimes
{\mathcal S}), \ \ \ \theta_2\,:=\, ({\rm Id}_E\otimes\widehat{\rho}_2)(\theta) \, \in\,
H^0(X,\, \text{End}(E)\otimes {\mathcal L}),
\end{equation}
where $\widehat{\rho}_1$ and $\widehat{\rho}_2$ are constructed in \eqref{e6d};
here $\theta$ is considered as a homomorphism from $E$ to $E\otimes V$. The given condition that
$\theta\bigwedge\theta\,=\, 0$ is equivalent to the following condition:
\begin{equation}\label{e7a}
(\theta_2\otimes {\rm Id}_{\mathcal S})\circ\theta_1\ =\
(\theta_1\otimes {\rm Id}_{\mathcal L})\circ\theta_2
\end{equation}
as homomorphisms from $E$ to $E\otimes\mathcal{S}\otimes\mathcal{L}$; here $\theta_1$ and $\theta_2$ are
considered as homomorphisms from $E$ to $E\otimes{\mathcal S}$ and $E\otimes{\mathcal L}$ respectively
(also $E\otimes\mathcal{S}\otimes\mathcal{L}$ is identified with $E\otimes\mathcal{L}\otimes\mathcal{S}$).

Take another $V$--twisted Higgs field $\theta'$ on $E$. Let $\theta'_1$ (respectively, $\theta'_2$)
be the $\mathcal S$--twisted (respectively, $\mathcal L$--twisted) Higgs field on $E$ given by
$\theta'$. If $\theta\, \not=\, \theta'$, then either $\theta_1\, \not=\, \theta'_1$ or
$\theta_2\, \not=\, \theta'_2$ or both hold. In other words, the above map
$\theta \ \longmapsto\ (\theta_1,\, \theta_2)$ is injective.

We will characterize all pairs $(\theta_1,\, \theta_2)$, where $\theta_1$ is a
$\mathcal S$--twisted Higgs field on $E$ and $\theta_2$ is a $\mathcal L$--twisted Higgs field on $E$,
that arise from $V$--twisted Higgs fields on $E$.

\section{Twisted Higgs bundles and spectral data}

\subsection{Decomposition of a $V$--twisted Higgs field}

Let
\begin{equation}\label{e8}
D\ =\ \sum_{i=1}^\ell x_i
\end{equation}
be the support of the torsion sheaf $\mathbb T$ in \eqref{e6}; recall that $D\,\leq\,{\mathbb D}$
(see \eqref{e2} for $\mathbb D$).

Take a holomorphic vector bundle $E$ on $X$. Let $\theta\, \in\, 
H^0(X,\, \text{End}(E)\otimes V)$ be a $V$--twisted Higgs field on $E$.
It gives the $\mathcal S$--twisted Higgs bundle $(E,\, \theta_1)$ and
the $\mathcal L$--twisted Higgs bundle $(E,\, \theta_2)$ (see \eqref{e7}).
Corresponding to $(E,\, \theta_1)$ we have a spectral curve
\begin{equation}\label{e9}
Y\ \subset\ {\mathcal S}
\end{equation}
and a torsionfree coherent sheaf
\begin{equation}\label{e10}
F \ \longrightarrow\ Y
\end{equation}
of rank one \cite{BNR}, \cite{Hi}. Let
\begin{equation}\label{e11}
\varphi\ :\ Y \ \longrightarrow\ X
\end{equation}
be the restriction of the natural projection ${\mathcal S}\, \longrightarrow\, X$.

We recall that $Y$ parametrizes the generalized eigenvalues of $\theta_1$, and $F$
is given by the generalized eigenspaces for $\theta_1$ \cite{BNR}, \cite{Hi}. In particular,
we have
\begin{equation}\label{e12}
F \ \subset\ \varphi^*E,
\end{equation}
and $\varphi^*\theta_1$ sends the subsheaf $F$ in \eqref{e12} to $F\otimes 
(\varphi^* {\mathcal S})\, \subset\, (\varphi^*E) \otimes (\varphi^* {\mathcal S})$.

\begin{lemma}\label{lem1}
\mbox{}
\begin{enumerate}
\item The homomorphism $\varphi^*\theta_2 \, :\, \varphi^*E\, \longrightarrow\, \varphi^*(E\otimes {\mathcal L})
\,=\, (\varphi^* E)\otimes (\varphi^* {\mathcal L})$, where $\theta_2$ is the $\mathcal L$--twisted
Higgs field on $E$ in \eqref{e7}, sends the subsheaf $F$ in \eqref{e12} to $F\otimes 
(\varphi^* {\mathcal L})\, \subset\, (\varphi^*E) \otimes (\varphi^* {\mathcal L})$.

\item The two homomorphisms $(\varphi^*\theta_1)\big\vert_F \, :\, F\, \longrightarrow\,
F\otimes (\varphi^* {\mathcal S})$ and $(\varphi^*\theta_2)\big\vert_F \, :\, F\, \longrightarrow\,
F\otimes (\varphi^* {\mathcal L})$ commute. In other words,
$$
(\varphi^*(\theta_2\otimes{\rm Id}_{\varphi^*\mathcal S})\big\vert_{F\otimes\varphi^*{\mathcal S}}
)\circ (\varphi^*\theta_1)\big\vert_F
\ =\ (\varphi^*(\theta_1\otimes {\rm Id}_{\varphi^*\mathcal L})\big\vert_{F\otimes\varphi^*{\mathcal L}})
\circ (\varphi^*\theta_2)\big\vert_F
$$
as homomorphisms from $F$ to $F\otimes \varphi^* ({\mathcal S}\otimes{\mathcal L})$.
\end{enumerate}
\end{lemma}

\begin{proof}
In view of \eqref{e7a}, the first statement follows immediately from the fact that if $A$ and $B$ are two
$r\times r$ matrices with complex entries with $AB\,=\, BA$, and $\lambda$ is a generalized eigenvalue of
$A$, then $B$ preserves the generalized eigenspace for the generalized eigenvalue $\lambda$ of $A$.

The second statement also follows from \eqref{e7a}.
\end{proof}

Take any point $x$ in $D$, which is the support of $\mathbb T$ (see \eqref{e8}). Next take any
connected component $y\, \in\, \varphi^{-1}(x)$ of the fiber, over $x$, of the
map $\varphi$ in \eqref{e11}. Note that the
reduced subscheme $y_{\rm red}$ is a single point, but $y$ need not be reduced.
From Lemma \ref{lem1}(1) we know that the homomorphism
$$
(\varphi^*\theta_2)_y \ :\ E_x \,=\, (\varphi^*E)_y \ \longrightarrow\ 
\varphi^*(E\otimes {\mathcal L})_y \,=\, E_x\otimes {\mathcal L}_x
$$
takes $F_y$ in \eqref{e12} to $F_y\otimes (\varphi^* {\mathcal L})_y \,=\, F_y\otimes{\mathcal L}_x$.
In other words, $(\theta_2)(x)$ produces a homomorphism
\begin{equation}\label{e13}
\theta_{2,y} \ :\ F_y \ \longrightarrow\ F_y\otimes{\mathcal L}_x.
\end{equation}

\begin{proposition}\label{prop1}
The homomorphism $\theta_{2,y}$ in \eqref{e13} has exactly one generalized eigenvalue. This
unique eigenvalue of $\theta_{2,y}$ is $-\rho^x (y_{\rm red})$, where $\rho^x\, : \, \mathcal{S}_x \longrightarrow \mathcal{L}_x$ is the homomorphism
in \eqref{e6c}. (Note that $y_{\rm red}$ is a point of the fiber of $\mathcal S$ over $x$ and
it lies on the spectral curve $Y$.)
\end{proposition}

\begin{proof}
Tensoring the exact sequence in \eqref{e6} with $\text{End}(E)$ we obtain the short
exact sequence
$$
0\, \longrightarrow\, \text{End}(E)\otimes V \,\xrightarrow{\,\,\,{\rm Id}_{\text{End}(E)}
\otimes\Phi\,\,\,}\,\text{End}(E)\otimes ({\mathcal S}\oplus {\mathcal L})\,
\stackrel{\Psi}{\longrightarrow}\,\text{End}(E)\otimes {\mathbb T} \, \longrightarrow\, 0.
$$
This gives an exact sequence
\[
0\ \longrightarrow\
H^0(X,\, \text{End}(E)\otimes V)\ \xrightarrow{\,\,\, ({\rm Id}_{\text{End}(E)}\otimes\Phi)_* \,\,\,}
\ H^0(X,\, \text{End}(E)\otimes ({\mathcal S}\oplus {\mathcal L}))
\]
\begin{equation}\label{e14}
\xrightarrow{\,\,\, \Psi_*\,\,\,} \ H^0(X,\, \text{End}(E)\otimes {\mathbb T}),
\end{equation}
where the homomorphisms $({\rm Id}_{\text{End}(E)}\otimes\Phi)_*$ and $\Psi_*$ are induced by
${\rm Id}_{\text{End}(E)}\otimes\Phi$ and $\Psi$ respectively.
Consider $\theta_1\oplus \theta_2\, \in\, H^0(X,\, \text{End}(E)\otimes ({\mathcal S}\oplus {\mathcal L}))$
(see \eqref{e7}). Since 
$$\theta_1\oplus \theta_2\ = \ ({\rm Id}_{\text{End}(E)}\otimes\Phi)_* (\theta),$$
it follows from \eqref{e14} that
\begin{equation}\label{e15}
\Psi_*(\theta_1\oplus \theta_2) \ = \ 0.
\end{equation}
For any $x\, \in\, D$, we have ${\mathbb T}_x\,=\, ({\mathcal S}_x\oplus {\mathcal L}_x)/
\Phi(V_x)$, where $\Phi$ is the homomorphism in \eqref{e6}. Let
\begin{equation}\label{exi}
\xi_{1,x}\, :\, {\mathcal S}_x \, \longrightarrow\, {\mathbb T}_x\ \ \, \text{ and }\ \ \,
\xi_{2,x}\, :\, {\mathcal L}_x \, \longrightarrow\, {\mathbb T}_x
\end{equation}
be the homomorphisms of fibers over $x$ given by $\xi$ in \eqref{e6}.
Denote the following compositions of homomorphisms
$$
V_x\,\xrightarrow{\,\,\,\Phi_x\,\,\,} \, {\mathcal S}_x\oplus {\mathcal L}_x
\, \longrightarrow\, {\mathcal S}_x \, \xrightarrow{\,\,\,\xi_{1,x}\,\,\,} \, {\mathbb T}_x,\ \text{and}
$$
$$
V_x\,\stackrel{\Phi_x}{\longrightarrow}\, {\mathcal S}_x\oplus {\mathcal L}_x
\, \longrightarrow\, {\mathcal L}_x \, \stackrel{\xi_{2,x}}{\longrightarrow}\, {\mathbb T}_x
$$
by $\alpha^x_1$ and $\alpha^x_2$ respectively. We have
\begin{equation}\label{exj}
\text{Id}_{{\rm End}(E)}\otimes\alpha^x_1
\,=\, \Psi_*(\theta_1)(x)\ \ \ \text{ and }\ \ \ \text{Id}_{{\rm End}(E)}\otimes\alpha^x_2
\,=\, \Psi_*(\theta_2)(x)
\end{equation}
as elements of $\text{End}(E_x)\otimes {\mathbb T}_x$,
where $\Psi_*$ is the homomorphism in \eqref{e14}. Consequently, the proposition follows from
\eqref{e15} and the construction (done in \eqref{e6c}) of the homomorphism $\rho^x$.
\end{proof}

As before, $y$ is a connected component of the fiber, over $x\,\in\, D$, of the map $\varphi$ in \eqref{e11}.
Let $$\theta_{1,y} \,:\, F_y \, \longrightarrow\, F_y\otimes (\varphi^*{\mathcal S})_y\,=\,
F_y\otimes{\mathcal S}_x$$ be the homomorphism given by $\theta_1(x)$. From \eqref{exj} we have
\begin{equation}\label{e-2}
({\rm Id}_{F_y}\otimes \xi_{1,x})\circ \theta_{1,y} \,:\, F_y \, \longrightarrow\, F_y\otimes{\mathbb T}_x\ \ 
\text{ and }\ \ ({\rm Id}_{F_y}\otimes \xi_{2,x})\circ \theta_{2,y} \,:\, F_y \,
\longrightarrow\, F_y\otimes{\mathbb T}_x,
\end{equation}
where $\xi_{1,x}$ and $\xi_{2,x}$ are the homomorphisms in \eqref{exi}, and $\theta_{2,y}$ is
defined in \eqref{e13}.

The following is an immediate consequence of \eqref{e15}.

\begin{corollary}\label{cor1}
The two homomorphisms $({\rm Id}_{F_y}\otimes\xi_{1,x})\circ \theta_{1,y}$ and $({\rm Id}_{F_y}\otimes\xi_{2,x})
\circ \theta_{2,y}$ in \eqref{e-2} from $F_y$ to $F_y\otimes{\mathbb T}_x$ satisfy the following equation:
$$
({\rm Id}_{F_y}\otimes\xi_{1,x})\circ \theta_{1,y} + ({\rm Id}_{F_y}\otimes\xi_{2,x})
\circ \theta_{2,y}\ =\ 0.
$$
\end{corollary}

Note that Proposition \ref{prop1} can be deduced from Corollary \ref{cor1}.

\subsection{Reconstructing $V$--twisted Higgs fields}

Take a holomorphic vector bundle $E$ on $X$ together with a $\mathcal S$--twisted Higgs field
\begin{equation}\label{e16}
\Theta \ \in\ H^0(X,\, \text{End}(E)\otimes{\mathcal S}).
\end{equation}
Let $Y\ \subset\ {\mathcal S}$ be the spectral curve and
\begin{equation}\label{e17}
F \ \longrightarrow\ Y
\end{equation}
the sheaf on $Y$ corresponding to $(E,\, \Theta)$. Let
\begin{equation}\label{e18}
\varphi\ :\ Y \ \longrightarrow\ X
\end{equation}
be the restriction of the natural projection ${\mathcal S}\, \longrightarrow\, X$. So we have
$F\, \subset\, \varphi^*E$. For any $x\, \in\, D$ (see \eqref{e8}), and any connected
component $y$ of $\varphi^{-1}(x)$, we have the homomorphism
\begin{equation}\label{e18a}
({\rm Id}_{F_y}\otimes \xi_{1,x})\circ\Theta_y\ :\ F_y\ \longrightarrow\ F_y\otimes {\mathbb T}_x,
\end{equation}
where $\xi_{1,x}$ is the homomorphism in \eqref{exi}.

Let
\begin{equation}\label{e19}
\Theta' \ \in\ H^0(X,\, \text{End}(E)\otimes{\mathcal L})
\end{equation}
be a $\mathcal L$--twisted Higgs field on $E$. As before,
take any point $x\, \in\, D$ and a connected component $y$ of $\varphi^{-1}(x)$.
If the homomorphism
$$
(\varphi^*\Theta')_y \ :\ E_x \,=\, (\varphi^*E)_y \ \longrightarrow\ 
\varphi^*(E\otimes {\mathcal L})_y \,=\, E_x\otimes {\mathcal L}_x
$$
takes $F_y\, \subset\, E_x$ in \eqref{e17} to $F_y\otimes (\varphi^* {\mathcal L})_y \,=\, F_y\otimes{\mathcal L}_x$,
then $\Theta'(x)$ produces a homomorphism
$$
\Theta'_y \ :\ F_y \ \longrightarrow\ F_y\otimes{\mathcal L}_x,
$$
in which case we have the homomorphism
\begin{equation}\label{e13b}
({\rm Id}_{F_y}\otimes \xi_{2,x})\circ\Theta'_y\ :\ F_y\ \longrightarrow\ F_y\otimes {\mathbb T}_x,
\end{equation}
where $\xi_{2,x}$ is the homomorphism in \eqref{exi}.

The following theorem shows that Lemma \ref{lem1} and Corollary \ref{cor1}
together characterize the $V$--twisted Higgs fields on $E$ in terms of the pair consisting of
a $\mathcal S$--twisted Higgs field and a $\mathcal L$--twisted Higgs field on $E$.

\begin{theorem}\label{thm1}
Take
$$
\Theta \ \in\ H^0(X,\, {\rm End}(E)\otimes{\mathcal S}) \ \ \text{ and }\ \
\Theta' \ \in\ H^0(X,\, {\rm End}(E)\otimes{\mathcal L})
$$
satisfying the following three conditions:
\begin{enumerate}
\item The homomorphism $\varphi^*\Theta' \, :\, \varphi^*E\, \longrightarrow\,
\, (\varphi^* E)\otimes \varphi^* {\mathcal L}$ sends the subsheaf $F$ in \eqref{e17} to $F\otimes 
\varphi^* {\mathcal L}\, \subset\, (\varphi^*E) \otimes \varphi^* {\mathcal L}$.

\item The two homomorphisms $(\varphi^*\Theta)\big\vert_F \, :\, F\, \longrightarrow\,
F\otimes \varphi^* {\mathcal S}$ and $(\varphi^*\Theta')\big\vert_F \, :\, F\, \longrightarrow\,
F\otimes \varphi^* {\mathcal L}$ commute. In other words,
$$
(\varphi^*(\Theta'\otimes{\rm Id}_{\varphi^*\mathcal S})\big\vert_{F\otimes\varphi^*{\mathcal S}}
)\circ (\varphi^*\Theta)\big\vert_F
\ =\ (\varphi^*(\Theta\otimes {\rm Id}_{\varphi^*\mathcal L})\big\vert_{F\otimes\varphi^*{\mathcal L}})
\circ (\varphi^*\Theta')\big\vert_F
$$
as homomorphisms from $F$ to $F\otimes \varphi^* ({\mathcal S}\otimes{\mathcal L})$.

\item For all $x$ and $y$ as above, the homomorphisms $({\rm Id}_{F_y}\otimes \xi_{1,x})\circ\Theta_y$
and $({\rm Id}_{F_y}\otimes \xi_{2,x})\circ\Theta'_y$
(see \eqref{e18a} and \eqref{e13b}) satisfy the equation
$$
({\rm Id}_{F_y}\otimes \xi_{1,x})\circ\Theta_y \, +\, ({\rm Id}_{F_y}\otimes \xi_{2,x})\circ\Theta'_y
\ =\ 0.
$$
\end{enumerate}
Then there is a unique $V$--twisted Higgs field $\theta$ on $E$ such that $\Theta$ and $\Theta'$ are
given by $\theta$ (as in \eqref{e7}).
\end{theorem}

\begin{proof}
The first two conditions imply that
\begin{equation}\label{c1}
(\Theta'\otimes {\rm Id}_{\mathcal S})\circ\Theta\ =\
(\Theta\otimes {\rm Id}_{\mathcal L})\circ\Theta'
\end{equation}
as homomorphisms from $E$ to $E\otimes\mathcal{S}\otimes\mathcal{L}$. In fact, \eqref{c1}
is equivalent to the first two conditions. From \eqref{c1} it follows that the homomorphism
$$
\Theta\oplus\Theta'\ :\ E\longrightarrow\ E\otimes ({\mathcal S}\oplus{\mathcal L})
$$
satisfies the equation
\begin{equation}\label{e20}
(\Theta\oplus\Theta')\bigwedge (\Theta\oplus\Theta')\ =\ 0.
\end{equation}

The third condition in the theorem implies that $\Psi_*(\Theta\oplus\Theta')\,=\, 0$,
where $\Psi_*$ is the homomorphism in \eqref{e14}. Hence from \eqref{e14} it follows
that there is a section
$$
\theta\, \in\, H^0(X,\, \text{End}(E)\otimes V)
$$
such that $({\rm Id}_{\text{End}(E)}\otimes\Phi)_*(\theta) \, =\, \Theta\oplus \Theta'$. Note
that this condition implies that $\theta$ is unique. Thus
$\Theta$ and $\Theta'$ are given by $\theta$ (as in \eqref{e7}). From \eqref{e20} it
follows that $\theta\bigwedge\theta\,=\, 0$ because $\theta$ and $\Theta\oplus\Theta'$ coincide
on $X\setminus D$.
\end{proof}

Lemma \ref{lem1}, Corollary \ref{cor1} and Theorem \ref{thm1} together give the following:

\begin{corollary}\label{cor2}
Let $E$ be a holomorphic vector bundle on $X$ and
$$
\Theta \ \in\ H^0(X,\, {\rm End}(E)\otimes{\mathcal S})\ \ \ and \ \ \
\Theta' \ \in\ H^0(X,\, {\rm End}(E)\otimes{\mathcal L}).
$$
Then there is a $V$--twisted Higgs field $\theta$ on $E$ such that $\Theta$ and $\Theta'$ are
given by $\theta$ (as in \eqref{e7}) if and only if the following three statements hold:
\begin{enumerate}
\item The homomorphism $\varphi^*\Theta' \, :\, \varphi^*E\, \longrightarrow\,
\, (\varphi^* E)\otimes \varphi^* {\mathcal L}$ sends the subsheaf $F$ in \eqref{e17} to $F\otimes 
\varphi^* {\mathcal L}\, \subset\, (\varphi^*E) \otimes \varphi^* {\mathcal L}$.

\item The two homomorphisms $(\varphi^*\Theta)\big\vert_F \, :\, F\, \longrightarrow\,
F\otimes \varphi^* {\mathcal S}$ and $(\varphi^*\Theta')\big\vert_F \, :\, F\, \longrightarrow\,
F\otimes \varphi^* {\mathcal L}$ commute. In other words,
$$
(\varphi^*(\Theta'\otimes{\rm Id}_{\varphi^*\mathcal S})\big\vert_{F\otimes\varphi^*{\mathcal S}}
)\circ (\varphi^*\Theta)\big\vert_F
\ =\ (\varphi^*(\Theta\otimes {\rm Id}_{\varphi^*\mathcal L})\big\vert_{F\otimes\varphi^*{\mathcal L}})
\circ (\varphi^*\Theta')\big\vert_F
$$
as homomorphisms from $F$ to $F\otimes \varphi^* ({\mathcal S}\otimes{\mathcal L})$.

\item For all $x$ and $y$ as before, the homomorphisms $({\rm Id}_{F_y}\otimes \xi_{1,x})\circ\Theta_y$
and $({\rm Id}_{F_y}\otimes \xi_{2,x})\circ\Theta'_y$
(see \eqref{e18a} and \eqref{e13b}) satisfy the equation
$$
({\rm Id}_{F_y}\otimes \xi_{1,x})\circ\Theta_y \, +\, ({\rm Id}_{F_y}\otimes \xi_{2,x})\circ\Theta'_y
\ =\ 0.
$$
\end{enumerate}
\end{corollary}

It was noted in the proof of Theorem \ref{thm1} that \eqref{c1}
is equivalent to the first two conditions in Theorem \ref{thm1}. Therefore, Corollary
\ref{cor2} gives the following:

\begin{corollary}\label{cor3}
Let $E$ be a holomorphic vector bundle on $X$ and
$$
\Theta \ \in\ H^0(X,\, {\rm End}(E)\otimes{\mathcal S})\ \ \ and \ \ \
\Theta' \ \in\ H^0(X,\, {\rm End}(E)\otimes{\mathcal L}).
$$
Then there is a $V$--twisted Higgs field $\theta$ on $E$ such that $\Theta$ and $\Theta'$ are
given by $\theta$ (as in \eqref{e7}) if and only if the following two statements hold:
\begin{enumerate}
\item $\Theta$ and $\Theta'$ commute, meaning
$$
(\Theta'\otimes {\rm Id}_{\mathcal S})\circ\Theta\ =\
(\Theta\otimes {\rm Id}_{\mathcal L})\circ\Theta'.
$$

\item For all $x$ and $y$ as before, the homomorphisms $({\rm Id}_{F_y}\otimes \xi_{1,x})\circ\Theta_y$
and $({\rm Id}_{F_y}\otimes \xi_{2,x})\circ\Theta'_y$
(see \eqref{e18a} and \eqref{e13b}) satisfy the equation
$$
({\rm Id}_{F_y}\otimes \xi_{1,x})\circ\Theta_y \, +\, ({\rm Id}_{F_y}\otimes \xi_{2,x})\circ\Theta'_y
\ =\ 0.
$$
\end{enumerate}
\end{corollary}

\section{Spectral construction}\label{subsection:spectral}

We can use Theorem \ref{thm1} and Corollaries \ref{cor2} and \ref{cor3} to describe a spectral construction 
for $V$--twisted Higgs bundles. Take a rank $2$ holomorphic vector bundle $V$ on $X$, and let $\mathcal S$ and 
$\mathcal L$ be holomorphic line bundles on $X$ constructed as in Section \ref{section:decomposition}, so 
that \eqref{e6} holds, in other words, there exists a torsion sheaf $\mathbb{T}$ supported on a reduced 
divisor $D$ such that
$$0\, \longrightarrow\, V \,\longrightarrow \,\mathcal{S}\oplus \mathcal{L} \, 
\longrightarrow\, \mathbb{T} \,\longrightarrow\, 0$$
is a short exact sequence.

Associate to each $V$--twisted rank $r$ Higgs bundle $(E,\, \theta)$ on $X$ the $\mathcal S$--twisted and
$\mathcal L$--twisted Higgs fields $\theta_1\,\in\, H^0(X,\, \operatorname{End}(E)\otimes {\mathcal S})$
and $\theta_2\,\in\, H^0(X,\, \operatorname{End}(E)\otimes{\mathcal L})$ respectively induced by $\theta$
(see \eqref{e7}). Following \cite{BNR}, let $s_i\,=\,\operatorname{tr}(\bigwedge^i \theta_1)
\,\in\, H^0(X,\, \mathcal{S}^{\otimes i})$, $1\, \leq\, i\, \leq\, r\,=\, {\rm rank}(E)$
be the coefficients of the characteristic polynomial of $\theta_1$, and let
$\mathcal{S} \, \supset \, X_s\, \longrightarrow \, X$ be the spectral curve associated to
$s\,=\,(s_i)_{i=1}^r$, defined by the characteristic polynomial
$$t^r-\phi^*s_1 t^{r-1}+\ldots+(-1)^r\phi^*s_r\ =\ 0,$$
where $\phi\, :\, \mathcal{S} \, \longrightarrow X$ is the natural projection and $t$ is the tautological
section of $\phi^*\mathcal{S}$ on the total space of $\mathcal{S}$.

\begin{theorem}\label{thm2}
Suppose that the curve $X_s$ is integral. Then there is a natural bijective correspondence
of the following type:
$$\left\{\begin{array}{c}
\text{Rank 1 torsion free }\varphi^*\mathcal{L}\text{--twisted Higgs}\\
\text{sheaves }(F,\,\theta_2')\text{ on } X_s, \text{where}\\
\theta_2'\, : \, F \, \longrightarrow\, F \otimes \varphi^*\mathcal{L}\\
\text{is such that}\\
\theta_{2,y}'\,=\,-\rho^x(y) \operatorname{id} \ \, \forall\,\, y \,\in\, \varphi^{-1}(D).
\end{array}
\right\} \quad \stackrel{1\,:\,1}{\longleftrightarrow} \quad \left\{\begin{array}{c}
V\text{--twisted Higgs bundles}\\
(E,\,\theta) \text{ with}\\
$$\theta\, :\, E\, \longrightarrow \, E \otimes V\\
\text{ such that the induced map }\\
\theta_1 \, : \, E\, \longrightarrow\, E\otimes \mathcal{S}\\
\text{has characteristic polynomial }$s$.
\end{array}
\right\}$$
\end{theorem}

\begin{proof}
Let $\theta \, : \, E\longrightarrow E\otimes V$ be a $V$--twisted Higgs field. Denote by
$\theta_1$ (respectively, $\theta_2$) the associated $\mathcal{S}$--twisted (respectively,
$\mathcal{L}$--twisted) Higgs field on $E$. Then $\theta_1$ induces a $\operatorname{Sym}
(\mathcal{S}^*)$--module structure on $E$, which factors through a $\varphi_*\mathcal{O}_{X_s} \,\cong
\, \operatorname{Sym}(\mathcal{S}^*)/\mathcal{I}$--module structure, where $\mathcal{I}$ is the ideal
generated by the characteristic polynomial of $\theta_1$ in $\operatorname{Sym}(\mathcal{S}^*)$. As a
consequence, if $X_s$ is integral then there exists a rank 1 torsionfree sheaf $F$ on $X_s$, satisfying
the condition $E\, =\,\varphi_*F$, such that the map $\theta_1$ is induced by the pushforward of the
$\mathcal{O}_{X_s}$--module structure on $F$.

Then, the twisted Higgs fields $\theta_1$ and $\theta_2$ commute if and only if the map $\theta_2 \,
: \, E\, \longrightarrow\, E \otimes \mathcal{L}$ is a map of $\operatorname{Sym}(\mathcal{S}^*)$--modules
in addition to being a map of $\mathcal{O}_X$--modules. Since $E\,=\,\varphi_*F$ is the pushforward of an
$\operatorname{Sym}(\mathcal{S}^*)$--module supported on $X_s$, we conclude that $\theta_1$ and $\theta_2$
commute if and only if $\theta_2\,=\,\varphi_*\theta_2'$ for some homomorphism
$$\theta_2' \, : \, F \, \longrightarrow \, F\otimes \varphi^*\mathcal{S}.$$
By construction, this map coincides with the restriction of $\varphi^*\theta_2$ to $F$ described by 
Theorem \ref{thm1}(1). By Corollary \ref{cor3} and Proposition \ref{prop1}, the two
Higgs commuting fields $\theta_1$ 
and $\theta_2$ are induced by a $V$--twisted Higgs field if and only if the unique eigenvalue of 
$\theta_{2,y}'$ is $-\rho^x(y)$ for each $y\,\in\, \varphi^{-1}(D)$. As $\theta_2'$ is a map of rank 1
torsionfree sheaves, and $X_s$ is assumed to be integral, this is equivalent to the condition
that $\theta_{2,y}'\,=\,-\rho^x(y) 
\operatorname{id}$. This completes the proof.
\end{proof}

Recall that a Higgs bundle $(E,\,\theta)$ is called \emph{stable} (respectively \emph{semistable}) if and only
if for each $0\, \neq\, E'\,\subsetneq\, E$ such that $\theta(E')\,\subseteq\, E'\otimes V$,
$$
\frac{\text{rank}(E')}{\text{degree(E')}}\, < \,\frac{\text{rank}(E)}{\text{degree}(E)}
\ \ \, \left(\text{respectively, }\, \frac{\text{rank}(E')}{\text{degree(E')}}\,
\leq \,\frac{\text{rank}(E)}{\text{degree}(E)}\right).$$

\begin{proposition}
Let $(E,\, \theta)$ be a $V$--twisted Higgs bundles, and let $(E,\, \theta_1)$ be its associated
$\mathcal{S}$--twisted Higgs bundle. If the spectral curve of $\theta_1$ is integral, then
$(E,\, \theta)$ does not admit any nontrivial $\theta$--invariant subbunde. Consequently,
$(E,\, \theta)$ is stable.
\end{proposition}

\begin{proof}
Let $\varphi \, : \, X_s\, \longrightarrow\, X$ be the spectral curve of $(E,\,\theta_1)$. Let
$0\,\neq\, E'\,\subsetneq\, E$ be a subbundle of the $V$--twisted Higgs bundle $(E,\theta)$. Denote by
$(F,\, \theta_2')$ the associated rank 1 torsionfree Higgs sheaf on $X_s$, given by the correspondence in
Theorem \ref{thm2}. It is clear by construction that $E'$ is preserved by $\theta$ if and only if
it is preserved simultaneously by $\theta_1$ and $\theta_2$. Now, $E'$ is preserved by $\theta_1$ if and only
if it is a sub--$\operatorname{Sym}(\mathcal{S}^*)$--module of $E$, so in that case $E'$ must be the
pushforward of some nontrivial subsheaf $F'$ of $F$. Since $X_s$ is integral, and the rank of $E'$ is less
than $r$, this is not possible.
\end{proof}


\end{document}